\documentclass[11pt, reqno]{amsart}

\usepackage{amssymb,geometry,stmaryrd,color}
\usepackage[T1]{fontenc}
\usepackage{dsfont}

\makeatletter
\@namedef{subjclassname@2020}{\textup{2020} Mathematics Subject Classification}
\makeatother

\evensidemargin\oddsidemargin

\begin{document}

\baselineskip=18pt
\setcounter{page}{1}
    
\newtheorem{Conj}{Conjecture}
\newtheorem{TheoA}{Theorem A\!\!}
\newtheorem{TheoB}{Theorem B\!\!}
\newtheorem{TheoC}{Theorem C\!\!}
\newtheorem{TheoD}{Theorem D\!\!}
\newtheorem{Theo}{Theorem}
\newtheorem{Lemm}{Lemma}
\newtheorem{Rem}{Remark}
\newtheorem{Ex}{Example}
\newtheorem{Def}{Definition}
\newtheorem{Coro}{Corollary}
\newtheorem{Propo}{Proposition}

\renewcommand{\theConj}{}
\renewcommand{\theTheoA}{}
\renewcommand{\theTheoB}{}
\renewcommand{\theTheoC}{}
\renewcommand{\theTheoD}{}

\def\a{\alpha}
\def\b{\beta}
\def\g{\gamma}
\def\B{{\bf B}} 
\def\CC{{\mathbb{C}}} 
\def\cG{{\mathcal{G}}} 
\def\cB{{\mathcal{B}}} 
\def\cI{{\mathcal{I}}} 
\def\cS{{\mathcal{S}}}
\def\UU{{\mathcal{U}}}
\def\ca{c_{\a}}
\def\ka{\kappa_{\a}}
\def\coa{c_{\a, 0}}
\def\cua{c_{\a, u}}
\def\cL{{\mathcal{L}}} 
\def\cW{{\mathcal{W}}} 
\def\Ea{E_\a}
\def\Eab{E_{\a,\b}}
\def\eps{{\varepsilon}} 
\def\esp{{\mathbb{E}}} 
\def\Ga{{\Gamma}} 
\def\G{{\bf G}} 
\def\K{{\bf K}}
\def\HH{{\bf H}}
\def\ii{{\rm i}}
\def\e{{\rm e}}
\def\L{{\bf L}}
\def\lbd{\lambda}
\def\lacc{\left\{}
\def\lcr{\left[}
\def\lpa{\left(}
\def\lva{\left|}
\def\M{{\bf M}}
\def\T{{\bf T}}
\def\Ma{\M_\a}
\def\Mab{\M_{\a,\b}}
\def\NN{{\mathbb{N}}} 
\def\pb{{\mathbb{P}}}
\def\pa{{\varphi_\a}}
\def\paa{{\varphi_{\a,1-\a}}} 
\def\pab{{\varphi_{\a,\b}}} 
\def\tpab{\hat{\varphi}_{a,b}} 
\def\tp{\tilde{p}}
\def\tq{\tilde{q}}
\def\hp{\hat{p}}
\def\tZ{\tilde{Z}} 
\def\tphi{\tilde{\varphi}} 
\def\tpsi{\tilde{\psi}} 
\def\rl{{\mathbb{R}}}
\def\racc{\right\}}
\def\rpa{\right)}
\def\rcr{\right]}
\def\rva{\right|}
\def\prost{{\succ_{\! st}}}
\def\W{{\bf W}}
\def\X{{\bf X}}
\def\Z{{\bf Z}}
\def\Xab{\X_{\a,\b}}
\def\XX{{\mathcal X}}
\def\Y{{\bf Y}}
\def\U{{\bf U}}
\def\V{{\bf V}}
\def\Un{{\bf 1}}
\def\ZZ{{\mathbb{Z}}}
\def\A{{\bf A}}
\def\AA{{\mathcal A}}
\def\hAA{{\hat \AA}}
\def\hL{{\hat L}}
\def\hT{{\hat T}}

\def\claw{\stackrel{d}{\longrightarrow}}
\def\elaw{\stackrel{d}{=}}
\def\qed{\hfill$\square$}

\newcommand*\pFqskip{8mu}
\catcode`,\active
\newcommand*\pFq{\begingroup
        \catcode`\,\active
        \def ,{\mskip\pFqskip\relax}%
        \dopFq
}
\catcode`\,12
\def\dopFq#1#2#3#4#5{%
        {}_{#1}F_{#2}\biggl[\genfrac..{0pt}{}{#3}{#4};#5\biggr]%
        \endgroup
}

\title{Persistence probabilities of autoregressive chains with continuous innovations}

\author[Titouan Donnart]{Titouan Donnart}

\address{Laboratoire Paul Painlev\'e, UMR 8524, Universit\'e de Lille, 42 rue Paul Duez, 59000 Lille, France. {\em Email}: {\tt titouan.donnart@univ-lille.fr}}

\author[Thomas Simon]{Thomas Simon}

\address{Laboratoire Paul Painlev\'e, UMR 8524, Universit\'e de Lille, 42 rue Paul Duez, 59000 Lille, France. {\em Email}: {\tt thomas.simon@univ-lille.fr}}



\begin{abstract} We consider the persistence probabilities of an autoregressive chain of order one with continuous  innovations. In the case of positive drifts, we show that these persistence probabilities are compound-geometric and satisfy a Baxter-Spitzer factorization generalizing that of the random walk. In the case of negative drifts, we exhibit a discrete Van Dantzig problem, which implies that the Baxter-Spitzer factorization never happens, except in a degenerate case. For positive drifts and log-concave innovations, we show that the first passage time in $(-\infty,0)$ has a log-convex distribution, whereas in the case of negative drifts and log-convex innovations on $\rl^+$, it has a log-concave distribution. The case of the bi-exponential innovations is studied in detail, which leads for positive drifts to an additive factorization of the exponential law.    

\end{abstract}

\maketitle

\section{Introduction and presentation of the results}
We consider in this paper a Markov chain $\{Z_n\}_{n \ge 0}$ defined by 
\begin{equation}
\label{MCMC}
Z_0\, = \, 0 \quad \text{and} \quad Z_n \, =\, \theta Z_{n-1} + X_n \quad \text{for } n \ge 1,
\end{equation}
where $\theta \in \mathbb{R}$ is called a drift parameter and $\{X_n\}_{n \ge 1}$ is a sequence of i.i.d. random variables whose common distribution $\mu$ is non-degenerate, and which we call the innovation sequence. This is an autoregressive chain of order one which can be viewed as a discrete version of the Ornstein-Uhlenbeck process and appears frequently in modelling as a particular instance of the ARMA process - see e.g. Chapter 3 in \cite{BD}. This is also a generalization of the random walk which corresponds to the case $\theta = 1.$ The driftless case $\theta =0$ where $\{Z_n\}_{n \ge 1}$ is itself i.i.d. will be implicitly excluded in the sequel. We study here the first passage time below zero
$$T_\theta\, =\, \inf\{n \ge 0, \; Z_{n+1} < 0\},$$
and its reliability function $p_n(\theta) = \mathbb{P}[T_\theta \ge n].$ Strictly speaking, the random variable $T_\theta$ is a shifted first passage time, which is however more convenient to formulate our results than the more common random variable $1+T_\theta = \inf\{n \ge 1, \; Z_{n} < 0\}.$  We will set $q_n(\theta) = \pb[T_\theta = n] = p_{n}(\theta) - p_{n+1}(\theta)$ for the (possibly defective) probability mass function of the random variable $T_\theta.$ The two generating functions
\begin{equation}
\label{Gene}
\varphi_\theta (z) \, =\, \sum_{n\ge 0} p_n(\theta) \, z^n\qquad\mbox{and}\qquad \psi_\theta (z) \, =\, \sum_{n\ge 0} q_n(\theta) \, z^n
\end{equation}
are well-defined on $(-1,1)$ and connected to one another by the formula
\begin{equation}
\label{Connect}
\varphi_\theta(z)\, =\, \frac{1-z\psi_\theta(z)}{1-z}\cdot
\end{equation}
Notice that for $n = 0,$ one has $p_0(\theta) = 1$ whereas for $n\ge 1,$ one has 
$$p_n(\theta)\, = \, \mathbb{P}[Z_1 \geqslant 0, \ldots, Z_n \geqslant  0].$$
The above quantity is often called in the literature a persistence probability, whose asymptotic behaviour as $n\to\infty$ has been well studied in recent years. Setting $X$ for the random variable having law $\mu,$ it is shown in Theorem 1 of \cite{HKW} that if $\pb[X > 0]\pb[X < 0] > 0$ and $\esp[\log (1 + \vert X\vert)] < \infty,$ then for all $\theta \in (0,1)$ there exists $\lbd_\theta \in (0,1]$ such that
\begin{equation}
\label{Rough}
\pb_x[T_\theta > n]^{1/n} \, \to\, \lbd_\theta, \qquad n\to \infty
\end{equation}
for all $x\ge 0,$ where $\pb_x$ stands for the law of $\{Z_n\}$ starting from $Z_0 = x.$ This logarithmic estimate can be refined as
\begin{equation}
\label{Fine}
\pb_x[T_\theta = n] \, \sim\, V(x)\, \lbd_\theta^{n}, \qquad n\to \infty
\end{equation}
for some positive function $V(x)$ and all $\theta \in (0,1),$ under some further assumptions on the innovation law $\mu$ - see Theorem 10 in \cite{HKW}. Moreover, the quantity $\lbd_\theta$ can be identified as the largest eigenvalue of some truncated operator built on the transition kernel of $\{Z_n\}$ - see Theorem 2.6 in \cite{AMZ}. See \cite{KN} for previous studies on the exponential boundedness of $T_\theta$ using a martingale approach. Let us also refer to \cite{AS} for a survey on persistence probabilities. 

In this paper, we will obtain some non-asymptotic results on the sequence $\{p_n(\theta)\}.$ We will need the general assumption that the innovation law $\mu$ has no atoms. This assumption is rather natural since if $\mu$ has an atom, then the support of $\{Z_n\}$ typically evolves with time because of the drift, which makes the exact study of $p_n(\theta)$ very complicated. See however the recent paper \cite{VW} for a precise asymptotic study when the innovations are Rademacher. In order to state our first main result, we need some further notation. Introduce the associate Markov chain $\{\tZ_n\}_{n \ge 0}$ defined by 
\[
\tZ_0\, = \, 0 \quad \text{and} \quad \tZ_n \, =\, \theta \tZ_{n-1} - X_n \quad \text{for } n \ge 1,
\]
and set $\{\tp_n(\theta)\}_{\{n\ge 0\}}$ and $\{\tq_n(\theta)\}_{\{n\ge 0\}}$ for the corresponding persistence and first passage probabilities, together with their generating functions $\tphi_\theta (z)$ and $\tpsi_\theta(z)$ as defined in \eqref{Gene}. In the following, when we state an identity between two entire series in $z$, we implicitly assume that the identity holds for all $z\in(-1,1),$ unless otherwise explicitly stated where it holds on a broader interval. 

\begin{TheoA} Assume $X$ has no atoms. Then, for all $\theta > 0$ one has
$$\varphi_\theta(z)\, =\, \frac{1}{1 - z\tpsi_{1/\theta}(z)}\cdot$$
\end{TheoA}

This identity has several interesting consequences. We first deduce that there exists a sequence of weights $\{a_n(\theta)\}_{n\ge 1}$ such that 
$$\varphi_\theta (z) \, =\, \exp \lcr \sum_{n\ge 1}\, \frac{a_n(\theta)}{n} \, z^n\rcr$$
where, taking logarithms,
\begin{equation}
\label{Loga}
\sum_{n\ge 1}\, \frac{a_n(\theta)}{n} \, z^n\, =\, \sum_{n\ge 1}\, \frac{(z\tpsi_{1/\theta}(z))^n}{n}\cdot
\end{equation}
The latter identity implies that $a_n(\theta) \ge 0$ for all $n\ge 1.$ Moreover, combining Theorem A and \eqref{Connect} yields the factorization 
\begin{equation}
\label{Duelle}
\varphi_\theta(z)\tphi_{1/\theta}(z)\, =\, \Big(\frac{1}{1 - z\psi_{\theta}(z)}\Big)\Big(\frac{1}{1 - z\tpsi_{1/\theta}(z)}\Big)\, =\, \frac{1}{1-z}
\end{equation}
so that $a_n(\theta) = 1- {\tilde a}_n(1/\theta)\in[0,1]$ for all $n\ge 1.$ Observe that in the degenerate case $\pb[X > 0] = 0$ one has $\varphi_\theta(z) = 1$ and $a_n(\theta) = 0$ for all $n\ge 1,$ whereas for $\pb[X < 0] = 0$ one has $\varphi_\theta(z) = 1/(1-z)$ and $a_n(\theta) = 1$ for all $n\ge 1.$ In the non-degenerate case $\pb[X> 0]\pb[X< 0]  = \tpsi_{1/\theta}(0)\psi_{1/\theta}(0) >  0,$ a direct consequence of \eqref{Loga} is that $a_n(\theta)\in\;]0,1[$ for all $n\ge 1.$ In the random walk case $\theta = 1,$ the weights can be exactly evaluated as 
$$a_n(1)\, =\, \pb[Z_n\ge 0]$$ 
for all $n\ge 1$ by the classical Baxter-Spitzer formula - see e.g. \cite{DMK} p.186. In the general case $\theta > 0,$ we will call the product formula \eqref{Duelle} with weights $a_n(\theta)\in [0,1]$ a Baxter-Spitzer factorization. In the symmetric case $\varphi_\theta = \tphi_{1/\theta}$, this factorization was already obtained in Theorem 1.4 of \cite{ABKS}, without the identification of each factor as reciprocal entire series, as an extension of the classical Sparre Andersen identity for random walks. Our argument to obtain \eqref{Duelle} in the general case is similar to that of \cite{ABKS}, with a use of telescopic sums which makes the proof more transparent.

In the positive recurrent case $\theta\in (0,1)$ and $\esp[\log (1 + \vert X\vert)] < \infty,$ one has $\varphi_\theta(1) = \esp[T_\theta] + 1 < \infty$ except in the degenerate case $\pb[X<0] = 0,$ and we can consider the random variable ${\hat T}_\theta$ with probability mass function
$$\pb[{\hat T}_\theta = n] \, =\, \frac{p_n(\theta)}{\varphi_\theta(1)}$$
for all $n\ge 0,$ which is naturally associated to the persistence probabilities, and which we call the tail random variable associated to $T_\theta.$ In this regard, Theorem A means that the random variable ${\hat T_\theta}$ is compound-geometric, that is it is distributed as a random walk on $\ZZ_+$ stopped at an independent geometric time. We refer to Chapter II.3-5 in \cite{SV} for more details on compound-geometric random variables, which form a subclass of the infinitely divisible (ID) random variables on $\ZZ_+$. 

For $\theta\in (0,1), \pb[X<0] > 0$ and $\esp[\log (1 + \vert X\vert)] < \infty,$ Theorem A implies that $\tpsi_{1/\theta}(1) < 1$ so that the random variable ${\tilde T}_{1/\theta}$ is defective and one cannot define its associated tail random variable. On the other hand, one has $\mu_\theta :=\inf\{z\ge 0, \;z\tpsi_{1/\theta}(z) \ge 1\} \ge 1$ and we can identify the logarithmic constant in \eqref{Rough} as
$$\lbd_\theta\, =\, \frac{1}{\mu_\theta}\cdot$$
When the singularity of $\varphi_\theta$ at $\mu_\theta$ is not essential\footnote{We believe that it is always the case.}, one has $\mu_\theta = \inf\{z\ge 0, \;z\tpsi_{1/\theta}(z) = 1\}$ and this establishes an intriguing correspondence between the eigenvalue $\lbd_\theta$ and the first positive root, which is simple by the absolute monotonicity of $z\tpsi_{1/\theta}(z),$ of the special function $z\mapsto z\tpsi_{1/\theta}(z) - 1.$

\medskip

Our second main result is a refinement of the compound-geometric property in the case of log-concave innovations. Observe that this situation contains the Gaussian innovations, which are constantly used in modelling with ARMA processes - see again  
Chapter 3 in \cite{BD}, and also \cite{AK247} for the related persistence probabilities.
 
\begin{TheoB} Assume $X$ has a log-concave density and that $\pb[X > 0]\pb[X < 0] > 0.$ Then, for all $\theta > 0$ the sequence $\{q_n(\theta)\}_{n\ge 0}$ is log-convex.
\end{TheoB}

The proof of Theorem B is independent of Theorem A and relies on the construction of a certain sequence of probability measures and the study of their stochastic ordering, which is ensured by the log-concavity condition on $X$. The reason why this result can be viewed a refinement of Theorem A comes from the well-known fact that the log-convexity of $\{q_n(\theta)\}_{n\ge 0}$ implies that of $\{p_n(\theta)\}_{n\ge 0},$ because
\begin{equation}
\label{LCLC}
p_n(\theta)p_{n+2}(\theta)\, -\, p_{n+1}(\theta)^2\, =\, \sum_{i\ge n+2} \Big(q_n(\theta)q_i(\theta)\, -\, q_{n+1}(\theta)q_{i-1}(\theta)\Big)
\end{equation}
is non-negative for all $n\ge 0.$ Indeed, since $p_1(\theta) = \pb[X>0] > 0,$  Kaluza's theorem on reciprocal entire series - see Satz 3 in \cite{Kaluza} - entails that the series
$$1\, -\, \frac{1}{\varphi_\theta(z)}$$
has non-negative coefficients, as shown in Theorem A which identifies this series with $z\tpsi_{1/\theta}(z)$. Observe that Kaluza's theorem and Theorem B also imply that
\begin{equation}
\label{Kalouze}
\psi_\theta (z) \, =\, \frac{\pb[X<0]}{1 -z \sigma_\theta(z)}
\end{equation}
for some non-negative series $\sigma_\theta(z),$ which means that the random variable $T_\theta$ itself is compound geometric. It is worth mentioning that for skip-free Markov chains on $\ZZ$, first passage time distributions to the nearest state are compound-geometric - see Theorem VII.2.1 in \cite{SV}. This property seems however to have been less frequently studied for first passage distributions of a Markov chain on a continuous state space, as is the case in the present paper.

\medskip

Our last main result handles the case $\theta < 0,$ which is very different. In this framework, the formula
$$p_n(\theta) \, = \, \int_0^\infty d\mu(x_1) \lpa \int_0^\infty\Un_{\{x_2\ge-\theta x_1\}}\,d\mu(x_2) \lpa \ldots \lpa \int_0^\infty \Un_{\{x_n\ge-(\theta^{n-1}x_1 +\cdots + x_{n-1})\}}\,d\mu(x_n)\right.\rpa \ldots \rpa$$
involves integrals on the positive orthant and the persistence probabilities will hence depend on the law of $X_+ = X\vee 0$ only. Moreover, if we set 
$$\rho\, =\, \pb[X\ge 0]$$
for the positivity parameter of $X$, then the above formula shows that $p_n(\theta) =\rho^n p^+_n(\theta)$ where $\{p^+_n(\theta)\}_{n\ge 0}$ are the persistence probabilities of the chain defined in \eqref{MCMC} with an innovation law given by that of $X^+ = X\vert X\ge 0.$ This reduces the study to the case $\rho = 1.$ Notice that in this case, there is a degenerate situation where 
\begin{equation}
\label{Kons}
p_n(\theta)\, =\, 1\qquad \mbox{for all $n\ge 0.$}
\end{equation}
It is easy to check - see Remark \ref{Misc1} (a) below - that this happens if and only if $\theta\ge -1$ and Supp $X\,\subset [c, C]$ for some $c,C\ge 0$ with $c+C\theta\ge 0.$ Observe also that in this degenerate situation one has $\tphi_{1/\theta}(z) =1$ so that the Baxter-Spitzer factorization is trivial. The following result states that the Baxter-Spitzer factorization cannot happen for $\theta < 0$ when \eqref{Kons} does not hold, since some weights $a_n(\theta)$ become negative. 
 
\begin{TheoC} Assume $\theta < 0,$ that \eqref{Kons} does not hold and that $X$ has no atoms. Then there exists a sequence of weights $\{a_n(\theta)\}_{n\ge 1}\in\rl$ with $a_n(\theta)\to 0$ as $n\to\infty,$ such that 
$$\varphi_\theta (z) \, =\, \exp \lcr \sum_{n\ge 1}\, \frac{a_n(\theta)}{n} \, z^n\rcr$$
for all $z\in [-1,1].$ If $\theta \in (-1,0)$ one has $a_2(\theta) = - a_2(1/\theta) > 0.$ The sequence $\{p_n(\theta)\}_{n\ge 0}$ is never log-convex. If $X$ has a log-convex density on $\rl^+,$ then the sequence $\{q_n(\theta)\}_{n\ge 1}$ is log-concave.
\end{TheoC}

When the drift is negative, it is easy to see - see Remark \ref{Surmu} (b) below - that the tail random variable ${\hat T}_\theta$ is always well-defined except in the degenerate case $\eqref{Kons}.$ The above result shows that ${\hat T}_\theta$ is never ID for $\theta < -1.$ We conjecture that ${\hat T}_\theta$ is ID for all $\theta \in [-1,0[$ - see Remark \ref{Misc1} (e) below - but this fact still eludes us. The main tool to prove Theorem C is the notion of quasi-infinite divisibility introduced in \cite{LPS}, and which applies here because of the other factorization
\begin{equation}
\label{Dantzig}
\varphi_\theta(z)\, \varphi_{1/\theta} (-z)\, =\, 1
\end{equation}
which was obtained in Theorem 1.2 of \cite{ABKS} when $X_+$ has no atoms. This factorization can be viewed as a discrete Van Dantzig problem and we comment on this curious question in Paragraph 3.2 below. In the last part of the paper we give thorough details on the case where $X$ has a Laplace or bi-exponential distribution, whose density is log-concave on $\rl$ and log-convex on $\rl^+$ and hence provides a good illustration of Theorems B and C. The elementary case with negative drift sheds some interesting light on the quasi-infinite divisibility of ${\hat T}_\theta$ and the underlying Van Dantzig problem. In the more involved case with positive drift, revisiting a computation made in \cite{Larralde} in the case $\theta\in (0,1)$ via $q-$series, we express all the spectral parameters involved in the exponential representation of $\varphi_\theta$ via the discrete set of zeroes of a certain transcendental function, and we show that the Baxter-Spitzer factorization reads then as an additive factorization of the exponential law.

\section{The case with positive drift}

\subsection{Proof of Theorem A} It follows from \eqref{Connect} that the result amounts to 
$$\varphi_\theta(z)\tphi_{1/\theta}(z)\, =\, \frac{1}{1-z}\cdot$$
Fixing $\theta > 0$ and setting $r=1/\theta,$ we hence need to show that 
$$\sum_{k=0}^{n} p_{n-k}(\theta) \, \tp_k(r)\, = \,1$$
for all $n\ge 0.$ The cases $n=0,1$ are obvious and we fix $n\ge 2.$ The proof starts as that of Theorem 1.4 in \cite{ABKS} and we will only point out the differences. Setting $F$ for the distribution function of $X\sim\mu$ and $G$ for that of $-X,$ using the notation of \cite{ABKS} for $A_n$ and $A_n(u)$ and starting from (38) therein leads to
\begin{eqnarray*}
p_n(\theta)\, = \, \pb[A_n] & = & \int_0^\infty \pb[A_{n-1}]\, dF(u_1)\, + \, \int_{-\infty}^0 \mathbb{P}\lcr A_{n-1}(-\theta^{1-n} u_1)\rcr \, dF(u_1) \\
& = & p_{n-1}(\theta) \, (1- \tilde{p}_1(1/\theta))\, + \, \int_0^{\infty} \mathbb{P}\lcr A_{n-1}(r^{n-1} u_1)\rcr \, dG(u_1).
\end{eqnarray*}
Therefore, 
$$p_n(\theta)\, + \, p_{n-1}(\theta)\tp_1(r)\, = \, p_{n-1} (\theta) \, + \, \int_0^{\infty} \mathbb{P}\lcr A_{n-1}(r^{n-1} u_1)\rcr \, dG(u_1).$$ If $n = 2,$ we compute
$$\int_0^{\infty} \mathbb{P}\lcr A_1(r u_1)\rcr \, dG(u_1)\, =\, \int_0^{\infty} \int_{r u_1}^{\infty} dF(u_2) dG(u_1)\, =\, \pb[ X_1 < 0, X_2 + r X_1 > 0] \, = \tp_1(r) - \tp_2 (r),$$
which yields
$$p_2(\theta)\, + \, p_1(\theta)\tp_1(r)\, + \, \tp_2(r)\, = \, p_1(\theta)\, +\, \tp_1(r)\, =\, 1$$
as required. If $n\ge 3,$ the expressions of $I_{n-1}$ after (38) in \cite{ABKS} and the previous computation imply 
\begin{eqnarray*}
\int_0^{\infty} \mathbb{P}\lcr A_{n-1}(\theta^{1-n} u_1)\rcr  dG(u_1) 
& = & p_{n-2}(\theta) \lpa \tp_1(r) - \tp_2 (r) \rpa\\
& & \qquad\, +\, \int_0^{\infty}\!\! \int_{-r u_1}^{\infty} \mathbb{P}(A_{n-2}(r^{n-1}u_1 + r^{n-2}u_2)) dG(u_2) dG(u_1),
\end{eqnarray*}
which leads altogether to
$$\sum_{k=0}^2 p_{n-k}(\theta) \, \tp_k(r)\, = \, \sum_{k=0}^1 p_{n-1-k}(\theta) \, \tp_k(r)\, +\, \int_0^{\infty}\!\! \int_{-r u_1}^{\infty} \mathbb{P}\lcr A_{n-2}(r^{n-1}u_1 + r^{n-2}u_2)\rcr dG(u_2) dG(u_1).$$
Iterating, we obtain
\begin{eqnarray*} 
\sum_{k=0}^{n-1} p_{n-k}(\theta) \, \tp_k(r) &= &\sum_{k=0}^{n-2} p_{n-1-k}(\theta) \tp_k(r)\, +\\ 
& & \int_0^{\infty} \int_{-ru_1}^{\infty} \cdots \int_{-(r^{n-1}u_1+ \cdots+r u_{n-2})}^{\infty}\!\!\!\!\!\!\!\!\!\!\!\!\!\!\!\!\!\!\! \mathbb{P}\lcr A_1(r^{n-1}u_1+\cdots+ u_{n-1})\rcr dG(u_{n-1}) \ldots dG(u_1)\\
& = & \sum_{k=0}^{n-2} p_{n-1-k}(\theta) \tp_k(r)\, + \, \pb [\tZ_1 > 0, \ldots,\tZ_{n-1} > 0, \tZ_ n < 0], 
\end{eqnarray*}
which leads finally to
$$\sum_{k=0}^{n} p_{n-k}(\theta) \, \tp_k(r) \, = \, \sum_{k=0}^{n-1} p_{n-1-k}(\theta) \tp_k(r)\, = \, \cdots\, =\, p_1(\theta) \, +\, \tp_1(r) \, =\, 1$$
as required.
\qed

\subsection{Some remarks on supermultiplicativity} Before proving Theorem B, we give some supermultiplicative properties related to the log-convexity of $\{p_n(\theta)\}_{n\ge 0},$ and true without any assumption on $\mu.$  

\begin{Propo}
\label{Surmulti}
    For every $z \ge 0$ and $m,n \ge 1$, one has 
    \[\pb_z[T_\theta > n+m]\; \ge \;  p_m(\theta)\, \pb_z[T_\theta > n].\]
\end{Propo}

\begin{proof}
Applying the simple Markov property a time $n$, we have
$$\pb_z[T_\theta > n+m]\, = \,\esp_z \left[ \Un_{\{T_\theta > n\}} \pb_{Z_n}[T_\theta > m] \right]\, \ge \, \,\esp_z \left[ \Un_{\{T_\theta > n\}} \pb_{0}[T_\theta > m] \right]\, =\, p_m(\theta)\, \pb_z[T_\theta > n]$$
where in the inequality we have used that $Z_n \ge 0$ on $\{T_\theta > n\}$ combined with the property, which is obvious by comparison, that the mapping
$$x\, \mapsto \, \pb_x[T_\theta > m]$$ 
is non-decreasing on $\rl^+$ because $\theta > 0$. 
\end{proof}

An immediate consequence is the following rough estimate, which recovers \eqref{Rough} for $z =0$ without any assumption on the innovation law $\mu.$ 

\begin{Coro}
\label{Surmultipli}
There exists $\lambda_\theta \in [0,1]$ such that: 
$$p_n(\theta)^{1/n} \, \to\, \lbd_\theta\qquad \mbox{as $n\to\infty.$}$$
\end{Coro}

\begin{proof}
Applying Proposition \ref{Surmulti} with $z = 0$ yields
\begin{equation}
\label{Surm}
p_{n+m}(\theta)\; \ge \;  p_m(\theta)\,p_n(\theta)
\end{equation}
for all $n,m\ge 1$ and the result follows from Fekete's lemma.
\end{proof}

\begin{Rem}
\label{Surmu} {\em (a) If $z > 0$ it is not true in general that 
\[\pb_z[T_\theta > n+m]\; \ge \;  \pb_z[T_\theta > m]\, \pb_z[T_\theta > n]\]
for all $m, n \ge 1.$ For example, if $\mu$ is uniform on $[-1,1],$ then for all $\theta \in (1/2,1)$ one has
$$\pb_2[T_\theta > 2] \, < \, 1 - \frac{1}{4}\lpa 1 - \frac{2\theta^2}{1+\theta}\rpa^2 \, <\, 1 \, =\, \pb_2[T_\theta > 1]^2.$$
On the other hand, using the notion of positive association - see \cite{Esary}, it is possible to complete the statement of Proposition \ref{Surmulti} with the following inequality 
\[\pb_z[T_\theta > n+m]\; \ge \; \pb_z[T_\theta > n]\, \pb_z\lcr Z_{n+1}\ge 0, \ldots. Z_{n+m}\ge 0\rcr\]
which relies only on $\pb_z.$ Indeed, under $\pb_z$ one has the decomposition
\[Z_k \, = \, \theta^k z + \sum_{j=1}^k \theta^{k-j} X_j.
\]
for all $k\ge 1,$ showing that the vector $\mathbf{Z} = (Z_1, \ldots, Z_{n+m})$ is a non-decreasing transformation of the 
vector $\mathbf{X} = (X_1, \ldots, X_n)$. Since the latter is i.i.d. and hence positively associated - see Theorem 2.1 in \cite{Esary}, the property conveys to $\mathbf{Z}$ and we can apply the association inequality to the non-decreasing functions $f(x_1, \ldots, x_{n+m}) = \Un_{\{x_1\ge 0, \ldots, x_n\ge 0\}}$ and $f(x_1, \ldots, x_{n+m}) = \Un_{\{x_{n+1}\ge 0, \ldots, x_{n+m}\ge 0\}}.$

\medskip

(b) If $\theta < 0,$ the same argument shows that $p_{n+m}(\theta) \le p_m(\theta)p_n(\theta)$ for all $m,n\ge 0$ since then the mapping $x\mapsto\pb_x[T_\theta > m]$ is non-increasing. We hence have again
$$p_n(\theta)^{1/n} \, \to\, \lbd_\theta\qquad \mbox{as $n\to\infty.$}$$
for some $\lbd_\theta\in [0,1]$ by Fekete's lemma. Observe also that if \eqref{Kons} does not hold, then  
$$\lbd_\theta \, =\, \exp\lcr \inf_{n\ge 1} \lpa\frac{\log (p_n(\theta))}{n}\rpa\rcr\, <\, 1.$$}
\end{Rem}

\subsection{Proof of Theorem B} We first observe that by log-concavity, the support of the density $f$ is an interval whose interior contains zero since $\pb[X > 0]\pb[X <0] > 0.$ This shows that the quantities 
$$q_n (\theta) \, = \, \pb[Z_1 \ge 0, \ldots, Z_n \ge 0, Z_{n+1} < 0]$$ 
are strictly positive for all $n \ge 0$. We need to show that the sequence  
$$n\, \mapsto\, r_n(\theta)\, =\, \frac{q_{n+1}(\theta)}{q_n(\theta)}$$
is non-decreasing. Since we are dealing with non strict inequalities, by approximation we may and will suppose that the density $f$ is positive on the whole $\rl.$ Introduce the functions
\begin{equation}
\label{KKK}
\kappa(x)\, =\, \int_0^\infty f(z-\theta x) \left( \int_{-\infty}^0 f(s-\theta z) \, ds \right) dz
\end{equation}
and
$$h(x)\, =\, \frac{1}{\kappa(x)} \int_0^\infty f(z-\theta x) \, \kappa(z) \, dz,$$
which are well-defined on $\rl $ by the log-concavity and positivity assumptions on $f.$ Observe that 
$$\kappa(x)\, =\, \pb_x\lcr Z_1 \ge 0, \, Z_2 < 0\rcr,$$ 
so that in particular, $q_1(\theta) = \kappa (0).$ 

\begin{Lemm}
\label{lem:g_non_decreasing}
The function $h$ is non-decreasing on $\mathbb{R}^+$.    
\end{Lemm}

\begin{proof}
We set
\begin{equation}
\label{RRR}
\rho(z) \, =\, \int_{-\infty}^0 f(s-\theta z) \, ds
\end{equation}
and first observe that the mapping
$$z\, \mapsto\, \frac{\kappa(z)}{\rho (z)}$$
is non-decreasing on $\rl^+.$ For all $0 < z_1 < z_2$ we have indeed, by Fubini's theorem, 
$$\kappa(z_2) \rho(z_1)\, - \,\kappa(z_1) \rho(z_2)\, = \,\int_0^\infty dt \, \rho(t) \lpa \int_{-\infty}^0 \lpa f(t-\theta z_2) f(s-\theta z_1) - f(t-\theta z_1) f(s-\theta z_2) \rpa ds\rpa$$
and the right-hand side is non-negative because
$$f(t-\theta z_2) f(s - \theta z_1)\, \ge\, f(t - \theta z_1) f(s - \theta z_2)$$
for all $s\le 0\le t$ by the log-concavity of $f.$ Now we can rewrite
$$ h(x)\, =\, \frac{\int_0^\infty f(z-\theta x) \kappa(z) \, dz}{\int_0^\infty f(z-\theta x) \rho(z) \, dz}$$
and again by Fubini's theorem the non-decreasing character of $h$ on $\rl^+$ amounts to
$$\int_0^\infty \int_0^\infty \Big( f(z-\theta x_2) f(w-\theta x_1) \kappa(z) \rho(w) - f(z-\theta x_1) f(w-\theta x_2) \rho(z) \kappa(w)\Big)\, dw dz\,  \ge \, 0$$
for all $0 < x_1 < x_2.$ But the latter is equivalent to
$$\int_{0 < w < z <\infty} \Big( f(z-\theta x_2) f(w-\theta x_1) - f(z-\theta x_1) f(w-\theta x_2) \Big)\lpa\kappa(z) \rho(w) - \rho(z) \kappa(w)\rpa dw dz\,  \ge \, 0,$$
which is true by the log-concavity of $f.$
\end{proof}

\begin{Rem}
\label{lem:growth_integral_ratio}
{\em The above proof shows that if $\rho_1, \rho_2$ are two positive functions on $\rl^+$ whose ratio $\rho_1/\rho_2$ is non-decreasing, the function
$$x\, \mapsto\, \frac{\int_0^\infty f(z-\theta x) \rho_1(z) \, dz}{\int_0^\infty f(z-\theta x) \rho_2(z) \, dz}$$
is also non-decreasing on $\rl^+,$ a fact which will be used subsequently.}
\end{Rem}

We will now express the involved ratio $r_n(\theta)$ as a certain integral of the function $h.$ Setting 
$$A_n\, =\, \{Z_1 \ge 0, \ldots, Z_n\ge 0\}$$
for all $n\ge 1,$ introduce a family of probability measures $\{\nu_n\}_{n\ge 0}$ on $\rl^+$ defined by
$$\nu_n (A)\, =\,  \frac{\mathbb{E}\lcr\Un_{A_n} \kappa(Z_n)  \Un_{\{Z_n \in A\}} \rcr}{\mathbb{E}\lcr\Un_{A_n}\kappa(Z_n)\rcr}$$
for all Borelian sets $A.$ Observe that $\nu_0 =\delta_0$ is the Dirac mass at zero.

\begin{Lemm}
    \label{pro:expression_ratio_qn}
For every $n \ge 1$ one has 
$$r_n(\theta) \, = \,\int_0^\infty h(y)\, \nu_{n-1}(dy).$$
\end{Lemm}
\begin{proof}
Fix $n\ge 1.$ By the Markov property, we have 
$$q_n (\theta)\, = \, \mathbb{E}\lcr \Un_{A_{n-1}} \mathbb{P}_{Z_{n-1}}\lcr Z_1 \ge 0, Z_2 < 0\rcr\rcr\, =\, \mathbb{E}\lcr \Un_{A_{n-1}} \kappa(Z_{n-1})\rcr$$
together with
\begin{eqnarray*}
q_{n+1}(\theta)\, = \, \mathbb{E}\lcr \Un_{A_{n}} \kappa(Z_{n})\rcr & = & \esp \lcr \lpa\frac{\kappa (X_n + \theta Z_{n-1})}{\kappa (Z_{n-1})}\Un_{\{X_n \ge -\theta Z_{n-1}\}} \rpa \Un_{A_{n-1}} \kappa(Z_{n-1})\rcr\\
& = & \esp \lcr \lpa\frac{1}{\kappa (Z_{n-1})}\int_0^\infty f(z- \theta Z_{n-1}) \kappa(z)\, dz \rpa \Un_{A_{n-1}} \kappa(Z_{n-1})\rcr\\
& = & \mathbb{E}\lcr \Un_{A_{n-1}} \kappa(Z_{n-1}) h(Z_{n-1})\rcr,
\end{eqnarray*}
which shows the result.
\end{proof}

We will now show that
\begin{equation}
\label{nuprec}
\nu_n\, \prec_{st} \, \nu_{n+1}
\end{equation}
for all $n\ge 0,$ where $\prec_{st}$ stands for the usual stochastic order. This is obvious for $n = 0$ since $\nu_0 =\delta_0,$ and for $n\ge 1$ this amounts to
$$\frac{\mathbb{E}\lcr\Un_{A_{n-1}} \kappa(Z_n)\Un_{\{Z_n \ge x\}}\rcr}{\mathbb{E}\lcr \Un_{A_n}\kappa(Z_n)\rcr}\, \le\, \frac{\mathbb{E}\lcr\Un_{A_n} \kappa(Z_{n+1})\Un_{\{Z_{n+1} \ge x\}}\rcr}{\mathbb{E}\lcr\Un_{A_{n+1}}\kappa(Z_{n+1})\rcr}\, =\, \frac{\int_0^\infty f(y) \,\mathbb{E}_y\lcr\Un_{A_{n-1}}\kappa(Z_{n})\Un_{Z_{n} \ge x}\rcr\, dy}{\int_0^\infty f(y) \,\mathbb{E}_y\lcr\Un_{A_n} \kappa(Z_{n})\rcr \, dy} $$
for all $x\ge 0.$ Therefore, it suffices to show that for all $n \ge 1$, the function:
    \[h_n : x \,\mapsto\, h_n(x)\, =\, \frac{\int_0^\infty f(y) \,\mathbb{E}_y\lcr\Un_{A_{n-1}}\kappa(Z_{n})\Un_{Z_{n} \ge x}\rcr\, dy}{\mathbb{E}\lcr\Un_{A_{n-1}} \kappa(Z_n)\Un_{\{Z_n \ge x\}}\rcr}   \]
    is non-decreasing on $\rl^+$. For $n = 1,$ one has by Fubini's theorem  
    \[ h_1(x)\, =\, \frac{\int_0^\infty f(y) \,\mathbb{E}_y\lcr\kappa(Z_{1})\Un_{Z_{1} \ge x}\rcr\, dy}{\mathbb{E}\lcr\kappa(Z_1)\Un_{\{Z_1 \ge x\}}\rcr}\, = \,\frac{\int_x^\infty \kappa(z) \left(\int_0^\infty f(y) f(z-\theta y) \, dy \right)\, dz}{\int_x^\infty \kappa(z) f(z) \, dz}\]
and the derivative in $x$ of the function on the right-hand side is non-negative on $\rl^+$ since
$$\kappa(x) \int_0^\infty f(y) \lpa \int_x^\infty \kappa(z) \lpa f(x) f(z-\theta y) - f(z) f(x-\theta y)\rpa dz\rpa  dy \,\ge 0$$
by the log-concavity of $f$. For $n \ge 2,$ one has
$$h_n(x)\, = \, \frac{\int_x^\infty \kappa(z) f_{n+1}(z) \, dz}{\int_x^\infty \kappa(z) f_n(z) \, dz}$$
with
$$f_n(z) \, = \, \int_0^\infty \!\!\cdots\!\! \int_0^\infty f(x_1) f(x_2 - \theta x_1)\,\cdots\, f(z-\theta x_{n-1}) \, dx_1\ldots dx_{n-1}.$$
Differentiating, we need to show that
    \[\kappa (x) \int_x^\infty \kappa(z) \lpa f_n(x) f_{n+1}(z) - f_{n+1}(x)f_n(z) \rpa \, dz \, \ge \,0,\]
    which is a consequence of the non-decreasing character of the function 
\begin{equation}
\label{Eqfn}
\frac{f_{n+1}(x)}{f_n(x)}\, =\, \frac{\int_0^\infty f(x-\theta z) f_n(z) \, dz}{\int_0^\infty f(x-\theta z) f_{n-1}(z) \, dz}
\end{equation}
on $\rl^+$, where we have set $f_1(x) = f(x).$ Finally, since again by the log-concavity of $f$ the mapping
    \[x\, \mapsto\, \frac{f_2 (x)}{f_1 (x)}\, = \, \int_0^\infty \left( \frac{f(x-\theta y)}{f(x)}\right) f(y) \, dy \]
is non-decreasing on $\rl^+,$ an induction based on \eqref{Eqfn} and a repeated use of Remark \ref{lem:growth_integral_ratio} completes the proof of \eqref{nuprec}. 

Combining Lemmas \ref{lem:g_non_decreasing}, Lemma \ref{pro:expression_ratio_qn} and \eqref{nuprec} shows that $r_{n+1}(\theta)\ge r_n(\theta)$ for all $n\ge 1$ and it remains to prove that $r_1(\theta)\ge r_0(\theta)$ in order to finish the proof of Theorem B. We compute 
\begin{equation}
\label{Caro}
r_0 (\theta)\, = \,\frac{\kappa(0)}{\rho (0)} \qquad \text{and} \qquad r_1(\theta)\, = \,\frac{\int_0^\infty f(z) \,\kappa(z) \, dz}{\int_0^\infty f(z) \rho(z) \, dz}
\end{equation}
with the above notation, and the result follows since $z\mapsto \kappa(z)/\rho(z)$ is non-decreasing, as shown during the proof of Lemma \ref{lem:g_non_decreasing}.

\qed

\begin{Rem} 

\label{BID}{\em (a) As mentioned in the introduction, the log-convexity of $\{q_n(\theta)\}_{n\ge 0}$ implies that of $\{p_n(\theta)\}_{n\ge 0},$ which is a refinement of \eqref{Surm}. Since the latter holds without the log-concavity assumption on the innovations, one may ask if this is not also true for the log-convexity. However, the above proof strongly depends on the log-concavity assumption on $f.$ 

\medskip

(b) If $\pb[X < 0]\pb[X > 0] > 0,$ then $p_1(\theta) > 0$ and Fekete's lemma shows that the constant $\lbd_\theta$ in Corollary  \ref{Surmultipli} is positive. If moreover $X$ has a log-concave density $f$, then Theorem B implies that the sequence $\{\lambda_\theta^{-n} p_n (\theta)\}_{n\ge 0}$ is log-convex and hence ultimately monotone and hence converges to some constant $c_\theta\in [0,\infty].$ Let us now discuss the positivity and finiteness of $c_\theta.$ If $\theta\in (0,1)$ and $f$ is furthermore positive on $\rl,$ then Theorem 10 in \cite{HKW} shows that $\lbd_\theta < 1$ and $c_\theta\in (0,\infty)$ - see also Theorem 1 in \cite{AK247} for a related result in the case of Gaussian innovations. If $\theta > 1,$ then \eqref{Duelle} implies by the same argument as in Proposition 4.5. of \cite{ABKS} that $\lbd_\theta = 1$ and that 
$$p_n(\theta)\,\to\, c_\theta \, =\, \frac{1}{\esp[T_{1/\theta}]}\, \in\, (0,1).$$
Recall finally that if $\theta =1,$ then $p_n(\theta) \sim \kappa n^{-1/2}$ as $n\to\infty$ for some $\kappa >0,$ since $\{Z_n\}$ is a random walk in the Gaussian domain of attraction, so that $c_1 = 0.$

\medskip

(c) For a given ID probability distribution $\{u_n\}$, some sufficient conditions for its log-convexity had been given in Theorem 2 of \cite{Hansen} in terms of the log-convexity of the corresponding sequence $\{a_n\}$ and the initial condition $a_1^2 \le a_2.$ The latter amounts to $u_1^2 \le u_2$ which is always satisfied by $u_n = p_n(\theta)$ in view of \eqref{Surm}. In this respect, one may ask if under the log-concavity assumption on $f$, the sequence of weights $\{a_n(\theta\}$ is not also log-convex. } 
\end{Rem}

\section{The case with negative drift}

\subsection{Proof of Theorem C} As mentioned in the introduction, we can suppose $\rho = 1.$ Since \eqref{Kons} does not hold, we know by Remark \ref{Surmu} (b) that the entire series $\varphi_\theta(z)$ has a radius of convergence $R_\theta > 1.$ Since $X$ has no atoms on $\rl^+,$ the factorization \eqref{Dantzig} holds true and is valid for all complex $z$ with $\vert z\vert < \min(R_\theta, R_{1/\theta}).$ If we consider the tail random variable ${\hat T}_\theta$ as defined in the introduction, then \eqref{Dantzig} implies that its characteristic function $t\mapsto \esp [e^{\ii t {\hat T}}]$ has no real zeroes and this shows by Theorem 8.1. in \cite{LPS} that ${\hat T}$ is quasi-infinitely divisible and that we have the required representation
\begin{equation}
\label{ExpR}
\varphi_\theta (z) \, =\, \exp \lcr \sum_{n\ge 1}\, \frac{a_n(\theta)}{n} \, z^n\rcr
\end{equation}
for some real sequence $\{a_n(\theta)\}_{n\ge 1}$, valid for all complex $z$ such that $\vert z\vert < R_\theta.$ Moreover, since $R_\theta > 1$ there exists $z > 1$ such that the positive series $\varphi_\theta (z)$ converges and this clearly implies $a_n(\theta)\to 0$ as $n\to \infty.$ One has $a_1(\theta) = a_1(1/\theta) = \pb[X > 0]$ and, if $\theta\in (-1,0),$ 
$$a_2(\theta) \, =\, \frac{a_2(\theta)\, -\,a_2(1/\theta)}{2}\, =\, p_2(\theta)\, - \, p_2(1/\theta)\, =\, 2\,\int_0^\infty dF(x) \lpa\int_{-\theta x}^{-x/\theta} dF(y)\rpa\, > \, 0,$$
where the first equality comes from \eqref{Dantzig}. This entails $a_2(\theta) = -a_2(1/\theta) > 0$ as required.

\medskip

The fact that $\{p_n(\theta)\}_{n\ge 0}$ is never log-convex is another consequence of Remark \ref{Surmu} (b) since the log-convexity would imply $p_{m+n}(\theta) = p_m(\theta)p_n(\theta)$ for all $m,n\ge 0$ and hence \eqref{Kons} since $p_1(\theta) = \rho =1,$ which is excluded by assumption. We finally show that the sequence
$$r_n(\theta)\, =\, \frac{q_{n+1}(\theta)}{q_n(\theta)}$$
is non-increasing for all $n\ge 1$ whenever $F$ has a log-convex density $f.$ Clearly, this property entails that $f$ is also positive and decreasing on $\rl^+,$ which implies that the function $\rho$ in \eqref{RRR} is positive increasing on $(0,\infty)$ whereas the function $\kappa$ in \eqref{KKK} is positive decreasing on $(0,\infty).$ Applying \eqref{Caro}, we first obtain $r_2(\theta) < r_1(\theta).$ Moreover, since $\theta < 0$ a perusal of the proofs of Lemmas \ref{lem:g_non_decreasing} and \ref{pro:expression_ratio_qn} shows for all $n\ge 1$ we have the same representation
$$r_{n+1}(\theta) \, = \,\int_0^\infty h(y)\, \nu_{n-1}(dy)$$
with $h$ non-increasing on $\rl^+$ and $\nu_{n-1}\prec_{st}\nu_n.$ This completes the argument.
\qed

\begin{Rem}
\label{Misc1}
{\em (a) If \eqref{Kons} holds, we have $\varphi_\theta(z)=1/(1-z)$ for all $z\in (-1,1)$ and $a_n(\theta) = 1$ for all $n\ge 1.$ As mentioned in the introduction, this happens if and only if $\theta\ge -1$ and Supp $X\,\subset [c, C]$ for some $c,C\ge 0$ with $c+C\theta\ge 0.$ If the condition holds, one has indeed
$$p_n(\theta) \, \ge \, \pb[X_n \ge (-\theta)X_{n-1}\ge \ldots \ge (-\theta)^{n-1} X_1 \ge 0]\, =\, 1$$ 
for all $n\ge 1,$ whereas if the condition fails one has 
$$p_2(\theta) \, \le \, \pb[X_2\ge -\theta X_1]\, <\, 1.$$
Observe also that if \eqref{Kons} holds, then $\varphi_{1/\theta}(z)=1+z$ for all $z\in (-1,1),$ so that ${\hat T}_{1/\theta}$ is Bernoulli with parameter 1/2, and $a_n(1/\theta) = (-1)^{n-1}$ for all $n\ge 1.$

\medskip

(b) The fact that $\{p_n(\theta)\}_{n\ge 0}$ is never log-convex can also be seen without Remark \ref{Surmu} (b) as a consequence of Kaluza's theorem. If indeed $\{p_n(\theta)\}_{n\ge 0}$ were log-convex, then a combination of \eqref{Dantzig} and Satz 3 in \cite{Kaluza} would show that there exists some non-negative sequence $\{b_n(\theta)\}_{n\ge 1}$ such that 
$$\varphi_{1/\theta} (-z)\, = \, \frac{1}{\varphi_\theta(z)}\, =\, 1\, -\, \sum_{n\ge 1} b_n(\theta) z^n$$
for all $z\in [-1,1].$ Hence, we would have $p_2(1/\theta) \le 0$ and $p_n(1/\theta) = 0$ for all $n\ge 2,$ which would imply $\varphi_{1/\theta}(z) = 1 + z$ and, using again \eqref{Dantzig}, that $p_n(\theta) = 1$ for all $n\ge 0.$ 

\medskip

(c) If $X$ has a log-convex density on $\rl^+,$ then Theorem C shows by \eqref{LCLC} that the sequence $\{p_n(\theta)\}_{n\ge 0}$ is also log-concave. Observe that here, Supp $X_+$ contains 0 and hence $\{p_n(\theta)\}_{n\ge 0}$ is not log-convex and hence not geometric. Observe also that \eqref{Dantzig} and Theorem 8.1.2 in \cite{Karlin} with $r = 2$ imply for all $\theta < 0$ that if $X$ has no atoms on $\rl^+,$ then
$$\{p_n(\theta)\}_{n\ge 0} \quad\mbox{is log-concave}\quad\Longleftrightarrow\quad \{p_n(1/\theta)\}_{n\ge 0} \quad\mbox{is log-concave.}$$

\medskip

(d) If $X$ has no atoms, the above result shows that the sequence $\{a_n(\theta)\}_{n\ge 1}$ appearing in \eqref{ExpR} takes negative values for $\theta < -1,$ so that the tail random variable ${\hat T}_\theta$ is quasi-infinitely divisible but never infinitely divisible. This fact can be illustrated more precisely in two explicit examples. 

\begin{itemize}

\item If $X_+$ has density $\lbd a^{-1}\Un_{[0,a]}(x)$ for some $a>0$ and $\lbd \in(0, 1),$ then the proof of Proposition 5.6 in \cite{ABKS} shows that 
$$a_n(\theta) \, =\, \frac{(-1)^{n-1}\, \lbd^n\, J_n(1/\theta) \, (1+ \cdots + \theta^{1-n})}{(n-1)!}$$
for all $\theta < -1$ and $n\ge 1,$ where $J_n$ is the $n$-th Mallows-Riordan polynomial. Since $J_n(1/\theta) > 0$ for all $\theta < -1$ and $n \ge 1$ - see e.g. Formula (3) in \cite{ABKS} for an explanation, the sequence $\{a_n(\theta)\}_{n\ge 1}$ strictly alternates, starting positive. For $\theta\in (-1,0),$ this implies that
$$a_n(\theta) \, =\, (-1)^{n-1} a_n(1/\theta)\, =\, \frac{\lbd^n\, J_n(\theta) \, (1+ \cdots + \theta^{n-1})}{(n-1)!}$$
is a positive sequence.

\item If $X_+$ has density $\lbd b e^{-bx}$ for some $a>0$ and $\lbd \in(0,1),$ then we will show in Paragraph \ref{BiExpM} below that 
$$a_n(\theta) \, =\, \lpa\frac{\lbd}{1-\theta}\rpa^n \Big( 1- \theta^n\Big),$$
which is positive if $\theta \in (-1,0)$ and alternates if $\theta <-1.$ 
\end{itemize} 
  
\medskip

(e) In case $\rho =1,$ we have $p_0(\theta) = p_1(\theta) = a_1(\theta) = 1.$ For all $n\ge 1,$ we also have the relationship 
$$(n+1)p_{n+1} (\theta) \, =\, \sum_{q=0}^n p_q(\theta) a_{n+1-q} (\theta)$$
which is obtained after taking the logarithmic derivative of \eqref{ExpR}. If we set $r_n(\theta) = p_n(\theta) - p_n(1/\theta),$ we have 
$$a_2(\theta)\, =\,r_2(\theta)\, = \,r_3(\theta)\, >\, 0 \qquad\mbox{and}\qquad a_4(\theta) = 2 r_4(\theta) \, -\, r_2(\theta)$$ for all $\theta \in (-1,0).$ The latter formula seems to imply that $a_4(\theta) > 0$ for all $\theta \in (-1,0),$ but we were unable to prove this. More generally, one may ask if $a_{2n}(\theta) > 0$ for all $\theta\in (-1,0)$ and $a_{2n+1}(\theta) \ge 0$ for all $\theta < 0.$ This would prove that ${\hat T}_\theta$ is infinitely divisible if and only if $\theta \ge -1.$}

\end{Rem} 

\subsection{On a discrete Van Dantzig problem} In this paragraph we give a few remarks to the following problem on probability generating functions which is motivated by \eqref{Dantzig}. Let $\{s_n\}_{n\ge 0}$ and $\{t_n\}_{n\ge 0}$ be two probability mass functions on $\NN$ and
$$f(z)\, =\, \sum_{n\ge 0} s_n \, z^n\qquad\mbox{and}\qquad g(z)\, =\, \sum_{n\ge 0} t_n \, z^n$$
be their generating series. We will say that $(f, g)$ is a discrete Van Dantzig (DVD) pair if the following identity holds
\begin{equation}
\label{VDM}
f(z)\,g(-z)\, =\, f(0)g(0)
\end{equation}
for all $z\in[-1,1].$ In case $f = g,$ we will say that $f$ is a self-reciprocal Van Dantzig generating function if $f(z)f(-z) = f(0)^2.$ The terminology comes from a famous problem on characteristic functions which was raised by D.~Van Dantzig: a pair $(f,g)$ of characteristic functions on $\rl$ is called a Van Dantzig pair if
$$f(t) g(\ii t)\, =\, 1$$
for all $t$ in some open neighbourhood of zero. Standard examples are the couples 
$$\lpa e^{-t^2/2}, e^{-t^2/2}\rpa, \qquad \lpa \cos t, \frac{1}{\cosh t}\rpa\qquad \mbox{and} \qquad\lpa\frac{\sin t}{t}, \frac{t}{\sinh t}\rpa.$$ 
Further examples and properties were studied in \cite{Luke}, and also later in \cite{RY} in the framework of infinitely divisible (ID) Wald couples. 

In the discrete framework, it is clear that the generating series of a probability mass function $\{s_n\}_{n\ge 0}$ on $\NN$ belongs to a DVD pair if and only if the (unique) solution $\{u_n\}_{n\ge 0}$ to the triangular array of equations 
$$u_0\, =\, 1 \qquad\mbox{and}\qquad\sum_{k=0}^{n} (-1)^k \, s_k\, u_{n-k}\, = \, 0, \qquad n\ge 1,$$
is a non-negative sequence. However, checking the latter condition may be non-trivial. Observe that necessarily, one must have $s_0 > 0$ since in \eqref{VDM} one has $g(z) > 0$ for all $z\in [0,1]$ and if one had $s_0=f(0)=0$ then the function $f(-z)$ would vanish on the whole $[0,1],$ which is impossible by analyticity. This shows, again by analyticity, that the identity
$$f(z)\,g(-z)\, =\, f(0)g(0)\, >\, 0$$
holds on the closed unit disk. In particular, Theorem 8.1. in \cite{LPS} shows that $\{s_n\}_{n\ge 0}$ must be quasi-ID, in other words for all $z$ in the closed unit disk one has
\begin{equation}
\label{QID}
f(z)\, =\, f(0)\, \exp\lcr \sum_{n\ge 1} \frac{a_n}{n}\, z^n\rcr
\end{equation}
for some real sequence $\{a_n\}_{n\ge 1}.$ In the self-reciprocal case, this yields immediately the following characterization.

\begin{Propo}
\label{SDVD}
A probability generating function $\varphi(z)$ on $\NN$ is self-reciprocal {\em DVD} if and only if there exists a real sequence $\{b_{2n+1}\}_{n\ge 0}$ such that
$$\varphi(z)\, =\, \varphi(0)\, \exp\lcr \sum_{n\ge 0} \frac{b_{2n+1}}{2n+1}\, z^{2n+1}\rcr.$$
\end{Propo}

At this point, it is worth mentioning that the question of the non-negativity of the coefficients of entire series constructed as exponential of real polynomials has been thoroughly investigated in \cite{LExp}. If $P$ is a real polynomial containing at most three monomials except the constant one, then it is not difficult to prove that 
$$\exp[P(z)]\;\mbox{has non-negative coefficients}\quad\Longleftrightarrow\quad P(z)\;\mbox{has non-negative coefficients.}$$ 
On the other hand, it follows from Theorem 1 in \cite{LExp} that if $p,q,r$ are odd integers such that $1<p<q<r$ and $q \,\wedge\, r = 1,$ then there exists $a > 0$ such that the entire series $\exp[z - a z^p + z^q + z^r]$ has non-negative coefficients; this example yields a self-reciprocal DVD probability generating function which is not ID.

The above Proposition \ref{SDVD} implies that an ID probability generating function is self-reciprocal DVD if and only if the underlying random variable is distributed as the independent sum
\begin{equation}
\label{Poisse}
\sum_{n\ge 0} (2n+1)\, X_n
\end{equation}
where $X_n$ has a Poisson distribution with parameter $b_{2n+1}/(2n+1)\ge 0$\footnote{We make the convention that $X_n = 0$ if $b_{2n+1} = 0$.} and $\sum_{n\ge 0} b_{2n+1} < \infty.$ This diversity in the discrete case contrasts with Theorem 3 in \cite{Luke}, which shows that the only ID self-reciprocal Van Dantzig characteristic function is the centered Gaussian. A large family of generating series corresponding to \eqref{Poisse} is 
\begin{equation}
\label{Polya}
f(z)\, =\, f(0)\,e^{\gamma z} \prod_{i\ge 1} \lpa\frac{1+c_i z}{1-c_i z}\rpa
\end{equation}
with $\gamma \ge 0$ and $c_i\in [0,1)$ such that $\sum_{i\ge 1} c_i < \infty,$ in which case one has
$$b_0 \, =\, \gamma\, +\, 2 \,\sum_{i\ge 1} c_i\qquad \mbox{and}\qquad b_{2n+1} \, =\, 2\, \sum_{i\ge 1} c_i^{2n+1}, \;\; n\ge 1.$$ 
Observe that $b_{2n+1} > 0$ for all $n\ge 0$ as soon as one $c_i$ is not zero. Recall also by Edrei's theorem - see e.g. Chapter 8 in \cite{Karlin}, that the underlying probability mass function is a P\'olya frequency sequence and is hence log-concave. The special case $c_1 = c \in [0,1/2]$ and $\gamma = c_n = 0$ for all $n\ge 2$ correspond to the persistence probabilities $\{p_n(-1)\}_{n\ge 0}$ where $X$ has density $2c e^{-x}$ on $\rl^+$ and we will come back to this example in Paragraph \ref{BiExpM} below. Another explicit example of self-reciprocal DVD generating series related to persistence probabilities is
$$f(z)\, =\, \frac{1+\sin(cz)}{\cos(cz)}$$
for some $c\in [0,1],$ which corresponds to $\{p_n(-1)\}_{n\ge 0}$ where $X$ has density $c \Un_{[0,1]}(x)$ on $\rl^+$ - see Remarks 2.6 and 5.7 in \cite{ABKS}. Here, one has $b_n = A_{2n} c^{2n+1}/(2n)!$ where $\{A_n\}_{n\ge 0}$ is the sequence of Euler's zigzag numbers. Observe that this example does not fall into the realm of \eqref{Polya} for $c > 0$ since $f$ has a negative pole at $z = -3\pi/(2c).$ 

\medskip

We now consider DVD pairs which are not self-reciprocal. An immediate consequence of \eqref{QID} is the following result, which mimics the Corollary p.119 in \cite{Luke} on the original Van Dantzig problem.  

\begin{Propo}
\label{NID}
If $(f, g)$ is a {\em DVD} pair such that $f\neq g$ and $f$ is {\em ID}, then $g$ is not {\em ID}. 
\end{Propo}

As in the self-reciprocal case, a large family of DVD pairs is given by 
$$\varphi(z)\, =\, \varphi(0)\,e^{\gamma z} \prod_{i\ge 1} \lpa\frac{1+c_i z}{1-d_i z}\rpa\qquad\mbox{and}\qquad \psi(z)\, =\, \psi(0)\,e^{\gamma z} \prod_{i\ge 1} \lpa\frac{1+d_i z}{1-c_i z}\rpa$$
with $\gamma \ge 0$ and $c_i, d_i\in [0,1)$ such that $\sum_{i\ge 1} (c_i + d_i) < \infty.$ It is worth mentioning that this family provides some examples of non self-reciprocal DVD pairs $(\varphi,\psi)$ such that both $\varphi$ and $\psi$ are not ID. If we choose $c_1 = c_2 = 1/\sqrt{3}, d_1 = 1/\sqrt{7}, d_2 = 1/\sqrt{2}$ and $\gamma=c_n=d_n=0$ for all $n\ge 3,$ then it is easy to check that neither $\varphi$ nor $\psi$ is ID since $c_1^2+c_2^2 > d_1^2 + d_2^2$ and $c_1^4 + c_2^4 < d_1^4 + d_2^4.$  

In general, it is not easy to prove that a given distribution on the integers belongs to some discrete Van Dantzig pair or not. Kaluza's aforementioned theorem provides some negative answers and with this result one can check that Borel, Sibuya, Yule-Simon or Zeta distributions do not belong to any DVD pair. Let us however conclude this paragraph with a positive answer. For every $N\ge 1$ and $p\in(0,1),$ the DVD pair
$$\varphi(z)\, =\, \lpa\frac{1-p}{1-pz}\rpa^N\qquad\mbox{and}\qquad \psi(z) \, =\, \lpa\frac{1+pz}{1+p}\rpa^N$$
involves indeed the negative binomial distribution with parameters $(N,1-p)$ for $\varphi$ and the binomial distribution with parameters $(N,p/(1+p))$ for $\psi.$ Observe that $p/(1+p) < 1/2$ and it is easy to check that a binomial distribution with success probability greater than $1/2$ cannot belong to a DVD pair. Recall also from Remark \ref{Misc1} (a) that for $N=1$ the above DVD pair corresponds to some $(\varphi_\theta, \varphi_{1/\theta})$ for some $\theta \in (-1,0).$

\section{Some explicit computations for Laplace innovations}

In this section we will provide some exact formulae in the case where the innovation $X$ has a density on $\rl$ given by
$$f(x)\, =\, (1-\lbd)ae^{-a\vert x\vert}\Un_{\{x <0\}}\, +\, \rho b e^{-b x}\Un_{\{x >0\}}$$
for some $a,b >0$ and $\rho = \pb[X> 0] \in (0,1].$ This section can be viewed as a counterpart to the recent paper \cite{AR} which investigates the persistence probabilities of moving average processes with such Laplace innovations. 
 
\subsection{The case with negative drift}
\label{BiExpM}

For $\theta < 0,$ the persistence probabilities are easily computed recursively as $p_0(\theta) = 1$ and
\begin{eqnarray*}
p_n(\theta) & = & \int_0^\infty  f(x_1)\, dx_1 \lpa \int_{-\theta x_1}^\infty f(x_2)\,dx_2 \lpa \ldots \lpa \int_{-(\theta^{n-1}x_1 +\cdots + x_{n-1})}^\infty f(x_n)\, dx_n\right.\rpa \ldots \rpa\\
& = & \int_0^\infty \ldots \int_0^\infty f(y_1) f(y_2 -\theta y_1)\ldots f(y_n - \theta y_{n-1})\, dy_1 \ldots dy_n\\
& = & \frac{\rho^n}{(1-\theta)^{n-1}}
\end{eqnarray*}
for all $n\ge 1.$ This implies 
$$\varphi_\theta(z)\, =\, \frac{1 - \theta (1+\rho z)}{1 - (\theta + \rho z)}$$
for $\vert z\vert < (1-\theta)/\rho,$ and the Van Dantzig identity $\varphi_\theta(z)\varphi_{1/\theta}(-z) = 1$ is elementarily seen. The weights are explicitly computed as
\begin{equation}
\label{Weigh}
a_n(\theta) \, =\, \lpa\frac{\rho}{1-\theta}\rpa^n \Big( 1- \theta^n\Big)
\end{equation}
for all $n\ge 1,$ which shows that the renewal random variable ${\hat T}_\theta$ is ID if and only if $\theta \ge -1.$ Observe that in the self-reciprocal case $\theta = -1,$ one has $a_{2n}(-1) = 0$ and $\varphi_{-1}(z) = (2+z)/(2-z),$ as in Example II.11.15 in \cite{SV}.  
  
The mass function of $T_\theta$ is given by $q_0(\theta) = 1- \rho$ and
$$q_n(\theta) \, =\, \Big( 1 -\rho-\theta\Big)\lpa\frac{\rho}{1-\theta}\rpa^n, \quad n\ge 1.$$
This distribution, which can be viewed as a modified geometric distribution, is log-concave in accordance with Theorem C, since $q_n(\theta)^2 = q_{n-1}(\theta) q_{n+1} (\theta)$ for all $n\ge 2$ and
$$q_1(\theta)^2 \, -\, q_0(\theta)q_2(\theta) \, =\, -\theta\Big(1-\rho-\theta\Big) \lpa\frac{\rho}{1-\theta}\rpa^2 \, > \, 0.$$
We further compute the generating series 
\begin{equation}
\label{GenM}
\psi_\theta(z)\, =\,\lpa\frac{(1-\rho)(1-\theta) - \rho\theta z}{1-\theta -\rho z}\rpa\, =\, \lpa\frac{1-\theta- \rho}{1-\theta -\rho z}\rpa\times\lpa 1 -\frac{\rho\theta}{\rho+\theta -1}\, +\,\frac{\rho\theta z}{\rho +\theta -1}\rpa, 
\end{equation}
which shows that $T_\theta$ is distributed as the independent sum of a Bernoulli random variable with parameter $\rho\theta/(\rho +\theta -1)$ and of a geometric random variable with parameter $(1-\rho-\theta)/(1-\theta).$ See Example II.11.15 in \cite{SV} for some further details on such independent sums.

\begin{Propo}
\label{quid}
One has the equivalences
$$\psi_\theta\;\mbox{belongs to a {\em DVD} pair}\quad\Leftrightarrow\quad T_\theta \;\mbox{is quasi-{\em ID}}\quad\Leftrightarrow\quad\theta\rho\, > \, (1-\theta)(\rho-1).$$
Moreover, $T_\theta$ is {\em ID} if and only if $\theta \ge \rho -1.$ 
\end{Propo}

\begin{proof} If $\psi_\theta$ belongs to a DVD pair, then we have seen above that it must be quasi-ID and the latter property implies $\theta\rho > (1-\theta)(\rho -1)$ since otherwise $\psi_\theta$ would vanish on $[-1,0]$ and this would contradict \eqref{QID}. Moreover, if $\theta\rho > (1-\theta)(\rho-1)$ then $T_\theta$ is the independent sum of a negative binomial random variable and a Bernoulli random variable with parameter $\rho\theta/(\rho +\theta -1)< 1/2$, which both belong to a DVD pair by our above discussion, and the same property hence also holds for $T_\theta.$ In this case, we have the exponential representation
$$\psi_\theta(z)\, =\, (1-\rho) \exp\lcr \sum_{n\ge 1} \frac{\rho^n}{(1-\theta)^n}\lpa 1 - \frac{\theta^n}{(1-\rho)^n}\rpa \frac{z^n}{n}\rcr,$$
showing that $T_\theta$ is ID if and only if $\theta \ge \rho -1.$ 

\end{proof}

\begin{Rem}
{\em (a) The above proof also shows that $\psi_\theta$ is self-reciprocal DVD if and only if $\theta = \rho -1$ and that the sequence $\{b_{2n+1}\}_{n\ge 0}$ in Proposition \ref{SDVD} is given by
$$b_{2n+1}\, =\, 2\lpa\frac{1+\theta}{1-\theta}\rpa^{2n+1}.$$

(b) The same argument shows that 
$$\varphi_\theta(z)\, =\, \exp\lcr \sum_{n\ge 1} \frac{\rho^n}{(1-\theta)^n}\Big( 1 - \theta^n\Big) \frac{z^n}{n}\rcr$$
always belongs to a DVD pair and that ${\hat T}_\theta$ is ID if and only if $\theta \ge -1.$ } 
\end{Rem}

\subsection{The case with positive drift}
\label{BiExpP}

We focus here on the symmetric case with $a = b > 0$ and $\lbd = 1/2.$ An easy scaling argument shows that the persistence probabilities do not depend on $a$ and we hence suppose that $X$ has density $e^{-\vert x\vert}/2.$ This case was studied in \cite{Larralde} in the case $\theta\in (0,1)$, where the following formula
\begin{equation}
\label{Larr}
\psi_\theta(z)\, = \,\frac{(\theta z, \theta^2)_\infty}{(\theta z, \theta^2)_\infty + (z, \theta^2)_\infty}
\end{equation}
was obtained, with the notation $(z,q)_\infty = \prod_{n\ge 0}(1-zq^n)$ for the $q-$Pochhammer symbol. See also \cite{MK} and the Appendix therein for a related physical model. Observe that in \cite{Larralde}, it is also checked that
$$\psi_\theta(z)\, \to\, \frac{1- \sqrt{1-z}}{z}\, =\, \psi_1 (z)$$
as $\theta \to 1.$ The following proposition computes $\psi_{\theta} (z)$ in the remaining case $\theta > 1.$ 

\begin{Propo}
\label{Lara}
One has 
$$\psi_{1/\theta}(z)\, = \,\frac{(\theta^2 z, \theta^2)_\infty}{(\theta^2 z, \theta^2)_\infty + (\theta z, \theta^2)_\infty}$$
for all $\theta \in (0,1).$  
\end{Propo}

\begin{proof} Fix $\theta\in (0,1)$ and set $F_i(z) = (\theta^i z, \theta^2)_\infty$ for $i=0,1,2.$ Using
$$\frac{F_0(z)}{F_2(z)}\, =\, 1-z,$$
we obtain
$$\lpa 1\, +\, \frac{F_0(z)}{F_1(z)}\rpa \lpa 1\, +\, \frac{F_1(z)}{F_2(z)}\rpa\, =\, 2\, +\, \frac{F_0(z)}{F_1(z)}\, +\, \frac{F_1(z)}{F_2(z)}\, -\, z$$
and then, using \eqref{Larr},
$$\psi_\theta(z)\, +\, \psi_\theta(\theta z)\, -\, z \,\psi_\theta(z)\psi_\theta(\theta z)\, =\, \lpa \frac{2\, +\, \frac{F_0(z)}{F_1(z)}\, +\, \frac{F_1(z)}{F_2(z)}\, -\, z}{\lpa 1\, +\, \frac{F_0(z)}{F_1(z)}\rpa \lpa 1\, +\, \frac{F_1(z)}{F_2(z)}\rpa}\rpa\, =\, 1.$$
Therefore, one has
$$(1 - z\psi_\theta(z))(1 -z\psi_\theta(\theta z))\, =\, 1-z$$
and a combination of \eqref{Connect} and \eqref{Duelle}, with $\tphi_{1/\theta} = \varphi_{1/\theta}$ by symmetry, implies 
$$\psi_{1/\theta}(z)\, = \,\psi_{\theta} (\theta z)\, =\, \frac{(\theta^2 z, \theta^2)_\infty}{(\theta^2 z, \theta^2)_\infty + (\theta z, \theta^2)_\infty}$$
as required.
\end{proof}

\begin{Rem}
\label{Finito}
{\em For all $\theta\in (0,1),$ one has $\psi_{1/\theta}(1) < 1$ so that $T_{1/\theta}$ is defective. The above formula yields
$$\pb[T_{1/\theta} = \infty]\, =\, 1\, -\,\psi_{1/\theta}(1)\, =\, \frac{(\theta, \theta^2)_\infty}{(\theta^2, \theta^2)_\infty + (\theta, \theta^2)_\infty}$$
and we can deduce from Theorem A that 
$$\esp[T_\theta]\, =\, \varphi_\theta (1)\, -\, 1\, =\, \frac{\psi_{1/\theta}(1)}{1\, -\,\psi_{1/\theta}(1)}\, =\, \frac{(\theta^2, \theta^2)_\infty}{(\theta, \theta^2)_\infty}\cdot$$
Applying Tonelli's theorem, we compute
$$\log\esp[T_\theta]\, =\, \sum_{n\ge 1} \frac{\theta^n}{n(1+\theta^n)}\,\sim\, -\frac{1}{2}\,\log(1-\theta)\qquad\mbox{as $\theta\to 1.$}$$
Some further work - see Theorem 3.2. in \cite{Banerjee} - yields the more precise asymptotic
$$\esp[T_\theta]\,\sim\, \sqrt{\frac{\pi}{1-\theta}}\qquad \mbox{as $\theta\to 1,$}$$
which was recently shown to hold true for a large class of centered innovations \cite{Rasch}.}
\end{Rem}

The next result gives a recurrent relation for the sequence $\{q_n(\theta)\}_{n\ge 1}$ which is reminiscent to that for the Catalan numbers.  

\begin{Propo}
\label{Recure}
For all $\theta > 0,$ one has $q_0(\theta) = 1/2$ and
$$q_n(\theta)\, =\, \frac{1}{1+\theta^n}\sum_{k=0}^{n-1} \theta^k q_k(\theta)q_{n-1-k}(\theta), \quad n\ge 1.$$
\end{Propo}

\begin{proof} First, one has $q_0(\theta) = 1 - p_1(\theta) = 1/2$ by the symmetry of $X.$ Moreover, we can write
$$\sum_{n\ge 1} (1+ \theta^n)\, q_n(\theta)\, z^n \, = \, \psi_\theta(z)\, +\, \psi_\theta(\theta z)\, -\, 1\, = \, z \psi_\theta(z)\psi_\theta(\theta z)\, = \, \sum_{n\ge 1} \lpa \sum_{k=0}^{n-1} \theta^k q_k(\theta)q_{n-1-k}(\theta) \rpa z^n,$$
where the second equality comes from the previous proof. This concludes the argument.

\end{proof}

\begin{Rem} {\em (a) For $\theta = 1$ the recursive formula becomes
$$q_n(1)\, =\,\frac{1}{2}\,\sum_{k=0}^{n-1} q_k(1)q_{n-k-1}(1)$$
and we recover the expression $q_n(1) = 2^{-(2n+1)} C_n$ where $C_n$ is the $n-$th Catalan number, coming also from $\psi_1(z) = (1 -\sqrt{1-z})/z.$

\medskip

(b) The recursive formula also implies the rational fraction representation
$$q_n(\theta)\, =\, \frac{1}{2^{n+1}}\, \frac{P_n(\theta)}{Q_n(\theta)}$$
where $P_n$ and $Q_n$ are monic palindromic polynomials with integer coefficients such that $\deg Q_n = \deg P_n +n.$ The first polynomials are
$$P_1 = P_2 = 1, \quad P_3 = 1+ 3X + X^2, \quad P_4 = 1+ 4X + 4X^2 + 4X^3 + X^4$$
and
$$Q_1 = 1+X, \quad Q_2 = 1+X^2, \quad Q_3 = (1+ X^2)(1+X^3), \quad Q_4 = (1+X)^2(1+X^2)(1+X^4).$$
Unfortunately, contrary to the case of uniform innovations - see \cite{ABKS}, we could not locate any relevant combinatorics lying behind those rational fractions.} 

\end{Rem}

We now study the infinite divisibility properties of the random variable $T_\theta.$ Since the symmetric bi-exponential distribution is log-concave, the sequence $\{q_n(\theta)\}_{\{n\ge 0\}}$ is log-convex by Theorem B, and Kaluza's theorem shows that its generating function is written as in \eqref{Kalouze} for some absolutely monotonic function $\sigma_\theta(z).$ 
For $\theta = 1$ one easily finds
$$\sigma_1(z)\, =\, \frac{1}{2(1+\sqrt{1-z})}\cdot$$
The following gives an expression of the function $\sigma_\theta(z)$ as a $q$-series in the remaining cases.

\begin{Propo}
\label{Compound}
With the above notation, for every $\theta\in(0,1)$ one has
$$\sigma_\theta(z)\, =\, \frac{1}{1+\theta}\,\sum_{k\ge 0} \frac{(\theta,\theta^2)_k}{(\theta^4;\theta^2)_k}\,(\theta z)^k\qquad\mbox{and}\qquad \sigma_{1/\theta}(z)\, =\, \frac{\theta}{1+\theta}\,\sum_{k\ge 0} \frac{(\theta,\theta^2)_k}{(\theta^4;\theta^2)_k}\,(\theta^2 z)^k.$$
\end{Propo}

\begin{proof} It follows from \eqref{Larr} that
$$\frac{1}{\psi_\theta(z)}\, =\, 1\,+\, \frac{(z,\theta^2)_\infty}{(\theta z; \theta^2)_\infty}\, =\, 1\, +\, \sum_{k\ge 0} \frac{(\theta^{-1},\theta^2)_k}{(\theta^2;\theta^2)_k}(\theta z)^k$$
where the last equality follows from the $q-$binomial theorem - see e.g. Formula (1.6) in \cite{Gasp} - with the usual notation $(z, q)_0 = 1$ and $(z,q)_k = \prod_{n= 0}^{k-1}(1-zq^n)$ for all $k\ge 1.$ On the other hand, one has  
$$\frac{(\theta^{-1},\theta^2)_k}{(\theta^2;\theta^2)_k}\,=\, - \frac{1}{\theta(1+\theta)} \lpa\frac{(\theta, \theta^2)_{k-1}}{(\theta^4;\theta^2)_{k-1}}\rpa$$
 for all $k \ge 1,$ which gives altogether the required formula 
$$\sigma_\theta(z)\, =\, \frac{1}{1+\theta}\,\sum_{k\ge 0} \frac{(\theta,\theta^2)_k}{(\theta^4;\theta^2)_k}\,(\theta z)^k.$$
Applying Proposition \ref{Lara}, we finally obtain 
$$\sigma_{1/\theta}(z)\, =\, \theta \sigma_{\theta}(\theta z)\, =\, \frac{\theta}{1+\theta}\,\sum_{k\ge 0} \frac{(\theta,\theta^2)_k}{(\theta^4;\theta^2)_k}\,(\theta^2 z)^k.$$ 
\end{proof}

We will now prove that the sequence $\{p_n(\theta)\}_{\{n\ge 0\}}$ can be represented as the integer moment sequence of a positive random variable. More precisely, we will show that there exists $X_\theta\in [0,1]$ such that
$$p_n(\theta)\, = \, \esp\lcr X_\theta^n\rcr, \qquad n\ge 0.$$
Equivalently, this means that both sequences $p_n(\theta)$ and $q_n(\theta) = \esp [X_\theta^n (1- X_\theta)]$ are completely monotonic and hence log-convex, which gives a refinement of Theorem B. By Theorem VI.7.8. in \cite{SV}, this also means that the law of $T_\theta$ is not only compound-geometric, but also a mixture of geometric distributions and we will come back to this property below. In the random walk case $\theta = 1,$ the property is well-known and easy to see because
$$\varphi_1(z) \, = \, \frac{1}{\sqrt{1-z}}\, =\, \esp\lcr e^{z\gamma_{1/2}}\rcr\, =\, \esp\lcr e^{z\gamma_1\times\beta_{1/2,1/2}}\rcr\, =\, \esp\lcr\frac{1}{1-z \beta_{1/2,1/2}}\rcr \, = \, \sum_{n\ge 0}\, \esp\lcr \beta_{1/2,1/2}^n\rcr z^n,$$ 
where $\gamma_t$ and $\beta_{a,b}$ stand for the usual gamma and beta random variables and the product in the third equality is independent. This means that $X_1\sim\beta_{1/2,1/2}$ has an arc-sine distribution and we refer to Example II.11.11 in \cite{SV} for further details. The following result shows that in the case $\theta\neq 1,$ the random variable $X_\theta$ is discrete.
																															\begin{Propo}
\label{Spektral}
For all $\theta\in (0,1)$ there exists a sequence $\{x_i(\theta), i\ge 0\}$ with $x_i(\theta) \in\; ]\theta^{2i+1}, \theta^{2i}[$ for all $i\ge 0,$ and two random variables $X_\theta$ and $X_{1/\theta}$ valued respectively in $\{x_i(\theta)\}$ and $\{\theta x_i(\theta)\}\cup\{1\},$ such that
$$p_n(\theta)\, =\, \esp\lcr X_\theta^n\rcr\qquad \mbox{and}\qquad p_n(1/\theta)\, =\, \esp[X_{1/\theta}^n]$$
for all $n\ge 0.$ 
\end{Propo}																																																												\begin{proof}

Fix $\theta\in (0,1)$ and consider the entire function
$$D_\theta(z) \,= \,(\theta z, \theta^2)_\infty\, + \,(z, \theta^2)_\infty.$$ 
It is claimed in \cite{Larralde} that the zeroes of this function are simple and real and that there is exactly one zero in each interval $]\theta^{-2i}, \theta^{-2i-1}[$ for all $i\ge 0.$ We will demonstrate this rigorously later on and accept it for the time being. Applying the $q-$binomial theorem - see (1.12) and (1.11) in \cite{Gasp} - implies
$$D_\theta(z) \,= \,\sum_{n\ge 0} \frac{\theta^{n(n-1)}(1 + \theta^n)}{(\theta^2,\theta^2)_n}\, (-z)^n$$
and shows that $D_\theta$ has order zero since
$$\frac{1}{n \log n} \log \lpa\frac{\theta^{n(n-1)}(1 + \theta^n)}{(\theta^2,\theta^2)_n}\rpa\, \to\, -\infty\qquad\mbox{as $n\to\infty.$ }$$  
By the Hadamard factorization theorem - see e.g. Theorem XI.3.4 in \cite{Conway}, we obtain
$$D_\theta (z) \,=\,  D_\theta(0)  \prod_{n\ge 1}\lpa 1-\frac{z}{z_n(\theta)}\rpa \,=\,  2 \prod_{n\ge 1}\lpa 1-\frac{z}{z_n(\theta)}\rpa$$
where $\{z_n(\theta)\}$ is the ordered sequence of zeroes of $D_\theta(z)$ and then, using \eqref{Larr}, 
$$\psi_\theta(z)\, =\, \frac{(\theta z, \theta^2)_\infty}{D_\theta(z)}\, =\, \frac{1}{2}\lim_{n\to\infty}  \prod_{i= 0}^n \lpa \frac{1- \theta^{2i+1} z}{1- x_i(\theta) z}\rpa$$
where we have set $x_i(\theta) = 1/z_i(\theta) \in\; ]\theta^{2i+1},\theta^{2i}[$ and the convergence is uniform on $[-1,1].$ This yields the partial fraction decomposition
$$\psi_\theta(z)\, =\, \frac{1}{2}\,\sum_{i\ge 0} \frac{c_i(\theta)}{1- x_i(\theta) z}$$
with 
$$c_i(\theta)\, =\, (\theta z_i(\theta), \theta^2)_\infty \, \prod_{j\neq i} \frac{1}{1-x_j(\theta)z_i(\theta)}\cdot$$
Moreover, decomposing
$$c_i(\theta)\, =\, (1-\theta^{2i+1}z_i(\theta))\,\times\, \prod_{j<i} \lpa\frac{1-\theta^{2j+1}z_i(\theta)}{1-x_j(\theta)z_i(\theta)}\rpa\,\times\, \prod_{j > i} \lpa\frac{1-\theta^{2j+1}z_i(\theta)}{1-x_j(\theta)z_i(\theta)}\rpa$$
shows that $c_i(\theta) \ge 0$ for all $i\ge 0$ since $z_i(\theta) \in \;]\theta^{-2i}, \theta^{-2i-1}[$ and both numerators and denominators are negative resp. positive in the finite product resp. in the infinite product. Therefore, by Tonelli's theorem,
$$\psi_\theta(z)\, =\, \frac{1}{2}\,\sum_{n\ge 0} \Big(\sum_{i\ge 0} c_i(\theta) (x_i(\theta))^n\Big)\, z^n$$
and we finally obtain
$$p_n(\theta)\, =\, \sum_{k\ge n} q_k(\theta)\, =\, \frac{1}{2}\,\sum_{i\ge 0} \frac{c_i(\theta)}{1-x_i(\theta)}\, (x_i(\theta))^n, \qquad n\ge 0.$$
Recalling that $p_0(\theta) = 1,$ this shows the required representation $p_n(\theta) = \esp[X_\theta^n]$ where $X_\theta$ is a random variable valued in $\{x_i(\theta)\}$ with
$$\pb[X_\theta = x_i(\theta)]\, =\, \frac{c_i(\theta)}{2(1-x_i(\theta))}\cdot$$
Using Remark \ref{Finito}, we get
\begin{eqnarray*}
p_n(1/\theta)\, =\, \pb[T_{1/\theta} =\infty] \, +\,\sum_{k\ge n} q_k(1/\theta) & = & \frac{(\theta, \theta^2)_\infty}{(\theta^2, \theta^2)_\infty + (\theta, \theta^2)_\infty} \, +\, \sum_{k\ge n} \theta^k q_k(\theta)\\
& = & \frac{(\theta, \theta^2)_\infty}{(\theta^2, \theta^2)_\infty + (\theta, \theta^2)_\infty} \, +\, \frac{1}{2}\,\sum_{i\ge 0} \frac{c_i(\theta)}{1-\theta x_i(\theta)}\, (\theta x_i(\theta))^n,
\end{eqnarray*}
which shows the required representation $p_n(1/\theta) = \esp[X_{1/\theta}^n]$ where $X_{1/\theta}$ is a random variable valued in $\{\theta x_i(\theta)\}\cup\{1\}$ with 
$$\pb[X_{1/\theta} = \theta x_i(\theta)]\, =\, \frac{c_i(\theta)}{2(1-\theta x_i(\theta))}\qquad\mbox{and}\qquad  \pb[X_{1/\theta} = 1]\, =\,\frac{(\theta, \theta^2)_\infty}{(\theta^2, \theta^2)_\infty + (\theta, \theta^2)_\infty}\cdot$$

\medskip

It remains to show the property for the function $D_\theta(z)$ which was mentioned at the beginning. The existence of one zero in each interval $]\theta^{-2i}, \theta^{-2i-1}[$ is guaranteed by the intermediate value theorem and the fact that $(-1)^i D_\theta(\theta^{-2i}) > 0$ and $(-1)^i D_\theta(\theta^{-2i-1}) < 0.$ We hence need to show that these zeroes are simple and that they are the only ones. To this end, fix $N > 0$ and consider a circle $\gamma_{N,\varepsilon}$ centered at the origin with radius $\theta^{-2N-1}+\varepsilon < \theta^{-2N-2},$ where $\varepsilon > 0$ is such that the function $D_\theta$ does not vanish on $\gamma_{N,\varepsilon}$. The latter condition is clearly possible by the principle of isolated zeros. For every $M \ge N,$ the polynomial function  
$$D_\theta^M(z)\, = \,\prod_{n=0}^M (1-\theta^{2n} z)\, + \, \prod_{n=0}^M(1-\theta^{2n+1} z)$$
vanishes on each interval $]\theta^{-2i}, \theta^{-2i-1}[$ for $i \in \{0,\ldots, M\}$ and thus possesses exactly $N+1$ zeroes inside the circle $\gamma_{N,\varepsilon}$. Setting $\eta = \min\{ \vert D_\theta(z)\vert, \; z\in\gamma_{N,\varepsilon}\} > 0,$ the uniform convergence of $D_\theta^M$ towards $D_\theta$ on compact sets as $M\to\infty$ shows that there exists $M_0 \ge N$ such that  
$$\lva D^M_\theta(z) - D_\theta(z)\rva \, <\, m\, \le \lva D_\theta(z)\rva$$
for all $M\ge M_0$ and $z\in\gamma_{N,\varepsilon}.$ 
Therefore, by Rouché's theorem - see e.g. Theorem V.3.8 in \cite{Conway}, the function $D_\theta$ has exactly $N+1$ zeroes counted with multiplicity inside the circle $\gamma_{N,\varepsilon},$ and we know that these zeroes lie in the required intervals. The result follows by letting $N \to \infty$.

\end{proof}

The above representation of $p_n(\theta)$ as the $n$-th moment of a discrete random variable gives a convergent series representation of $q_n(\theta)$ as $n\to\infty$ which considerably refines \eqref{Fine} in the case $x= 0.$ As seen indeed in the above proof, for every $\theta\in (0,1)$ one has
$$q_n(\theta)\, =\, \frac{1}{2}\,\sum_{i\ge 0} c_i(\theta) (x_i(\theta))^n\, \underset{n \to \infty}{\sim}\, \frac{c_0(\theta)}{2}\,(x_0(\theta))^n$$
and
$$q_n(1/\theta)\, =\, \frac{1}{2}\,\sum_{i\ge 0} c_i(\theta) (\theta x_i(\theta))^n\, \underset{n \to \infty}{\sim}\, \frac{c_0(\theta)}{2}\,(\theta x_0(\theta))^n,$$
showing that $c_0(\theta) = \lbd_\theta$ in \eqref{Fine}. These semi-explicit series representation come from the solvable character of the AR model with bi-exponential innovations. One may ask if this is not an example of a more general phenomenon for continuous innovations, in the case $\theta\neq 1.$ See Proposition 2.5 and Formula (26) in \cite{ABKS} for such a series representation in the case of uniform innovations. See also Theorem 10 in \cite{HKW} for a result in a more general framework which shows that if such a discrete random variable $X_\theta \in [0,1]$ exists, then for $\theta\in (0,1)$ the maximum $\lbd_\theta$ of its support must be isolated in $(0,1)$ and the singularity of $\varphi_\theta$ at $\mu_\theta$ is removable. Observe that by Theorem A, this claim amounts to showing that
\begin{equation}
\label{Stil}
\frac{1}{1- z\tpsi_{1/\theta}(z)}\, =\, \esp\lcr \frac{1}{1-zX_\theta}\rcr
\end{equation}
is the Cauchy-Stieltjes transform of this random variable $X_\theta$. We also might expect that the discrete support of $X_\theta$, which would form the "harmonics" of the persistence probabilities, is expressed in terms of the zeroes of a certain transcendental function. For uniform innovations this function is the deformed exponential - see Formula (6) in \cite{ABKS}, whereas for bi-exponential innovations it is the above function $D_\theta(z).$ 

In the random walk case $\theta = 1,$ the Baxter-Spitzer factorization implies that the random variable $X_1$ exists and is continuous if $\mathbb{P}[Z_n \ge 0]$ does not depend on $n,$ which happens when $X$ is symmetric continuous or strict stable\footnote{We are not aware of any other example.} with positivity parameter $\mathbb{P}(X \ge 0) = \rho.$ Here, one has 
$$\psi_1(z) \,= \, 1 \,-\,(1-z)^{1-\rho}  \qquad \text{and} \qquad q_n(1)\, = \,\frac{\sin(\pi p)}{\pi} \int_0^1 x^n \left(\frac{1-x}{x}\right)^{1-p} dx$$
so that $X_1$ has a generalized arcsine law, whose continuous character matches the expansion of the persistence probabilities at infinity which is algebraic. If $\pb[Z_n\ge 0]$ is not constant, it is however not clear to the authors whether such a continuous random variable $X_1$ always exists. 

Let us finally mention that combined to \eqref{Duelle}, the claimed Stieltjes representation \eqref{Stil} amounts to an additive factorization of the unit exponential random variable $\L$, which reads
$$\L \, \sim\, \L_\theta \, +\, {\tilde \L}_{1/\theta}$$
for all $\theta\in (0,1),$ where $\L_\theta \sim \L \times X_\theta$ and ${\tilde \L}_{1/\theta}\sim \L\times {\tilde X}_{1/\theta}$ are two exponential mixtures, and the sum is independent. Equivalently - see Proposition VI.3.5 in \cite{SV}, there exists two measurable functions $q_\theta$ and ${\tq}_{1/\theta}$ from $[1,\infty) \to [0,1]$ with $q_\theta + \tq_{1/\theta} = \Un_{[1,\infty)},$ such that
$$\varphi_\theta(z)\, =\, \exp \lcr\int_1^\infty \lpa \frac{1}{x-z} - \frac{1}{x}\rpa q_\theta(x)\, dx\rcr\quad\mbox{and}\quad \tphi_{1/\theta}(z)\, =\, \exp \lcr\int_1^\infty \lpa \frac{1}{x-z} - \frac{1}{x}\rpa \tq_{1/\theta}(x)\, dx\rcr.$$
Simple manipulations - see also Theorem VI.7.8 in \cite{SV} - imply then the moment representations
$$\frac{a_{n+1}(\theta)}{n+1}\, =\, \int_0^1 x^n \, q_\theta (1/x)\,dx \qquad \mbox{and}\qquad \frac{{\tilde a}_{n+1}(1/\theta)}{n+1}\, =\, \int_0^1 x^n\, \tq_{1/\theta} (1/x)\,dx$$
for all $n\ge 0,$ which would improve on \eqref{Loga}. In the case of symmetric bi-exponential innovations, our claim is true and we can compute the function $q_\theta$ explicitly in terms of the ordered sequence $\{z_i(\theta), i\ge 0\}$ of zeroes of the function $D_\theta(z).$  

\begin{Propo}
\label{LTZ}
With the above notations, one has
$$q_\theta\, =\, \sum_{i\ge 0} \Un_{[z_i(\theta), \theta^{-1}z_i(\theta)[}\qquad \mbox{and}\qquad q_{1/\theta}\, =\,\Un_{[1, z_0(\theta[}\, +\, \sum_{i\ge 0} \Un_{[\theta^{-1}z_i(\theta), z_{i+1}(\theta)[}$$
for all $\theta\in (0,1).$ 
\end{Propo}

\begin{proof}
Fixing $\theta \in (0,1),$ it is enough to show the formula for $q_\theta,$ which is easily seen by our previous discussion to be equivalent to the convergent product representation
$$\varphi_\theta(z) \, =\, \prod_{i\ge 0} \lpa\frac{1-\theta x_i(\theta)z}{1- x_i(\theta) z}\rpa$$
with $x_i(\theta) = 1/z_i(\theta) \in \;]\theta^{2i+1},\theta^{2i}[$ for all $i\ge 0.$ Combining Theorem A and Proposition  \ref{Lara} shows that 
$$\varphi_\theta(z) \, =\,\frac{1}{1-z\psi_{1/\theta}(z)}\, =\, \frac{D_\theta(\theta z)}{D_\theta(\theta z) - z(\theta^2z,\theta^2)_\infty}$$
is the quotient of two entire functions of order zero, which can be expressed as an infinite product by the Hadamard factorization theorem. We have seen during the proof of Proposition \ref{Spektral} that $D_\theta(\theta z)\, =\, \prod_{i\ge 0} (1-\theta x_i(\theta)z)$ and the factorization
$$D_\theta(\theta z) - z(\theta^2z,\theta^2)_\infty\, =\, \prod_{i\ge 0} \Big(1-x_i(\theta)z\Big)$$
comes from the fact that the zeroes of the function on the left hand side are the poles of the function $\varphi_\theta,$ which are given by the sequence $\{z_i(\theta)\}$ and are simple. This completes the argument.

\end{proof}

\begin{Rem}
\label{Weighs}
{\em The computation also implies
$$a_n(\theta) \, =\, \Big( 1- \theta^n\Big)\sum_{i\ge 0} x_i(\theta)^n\qquad \mbox{and} \qquad a_n(1/\theta) \, =\, 1\,+\,\Big(\theta^n -1\Big)\sum_{i\ge 0} x_i(\theta)^n$$
for all $\theta\in (0,1)$ and $n\ge 1,$ to be compared with \eqref{Weigh}.}
\end{Rem}

\section*{Acknowledgement}
Thanks go to Jean-François Burnol for his advice on Proposition \ref{Spektral}.

\end{document}